\documentclass[12pt]{amsart}
\usepackage[all]{xy}
\usepackage{graphicx}
\usepackage{tikz}  

\usepackage{amsmath}
\usepackage{amssymb}
\usepackage{amsfonts}
\usepackage{mathrsfs}
\usepackage{mathtools}

\newcommand{\Cdb}{\ensuremath{\mathbb{C}}}
\newcommand{\Ddb}{\ensuremath{\mathbb{D}}}

\newcommand{\Pdb}{\ensuremath{\mathbb{P}}}

\newcommand{\Rdb}{\ensuremath{\mathbb{R}}}

\newcommand{\Zdb}{\ensuremath{\mathbb{Z}}}

\newcommand{\D}{\mbox{${\mathcal D}$}}

\newcommand{\norm}[1]{\Vert#1\Vert}
\newcommand{\bignorm}[1]{\bigl\Vert#1\bigr\Vert}
\newcommand{\Bignorm}[1]{\Bigl\Vert#1\Bigr\Vert}

\newtheorem{theorem}{Theorem}[section]
\newtheorem{lemma}[theorem]{Lemma}
\newtheorem{corollary}[theorem]{Corollary}

\theoremstyle{remark}
\newtheorem{remark}[theorem]{\bf Remark}
\theoremstyle{definition}

\numberwithin{equation}{section}

\begin{document}

\title[Counterexamples $R$-boundedness]
{New counterexamples on Ritt operators, sectorial operators 
and $R$-boundedness}

\author[L. Arnold]{Loris Arnold}
\email{loris.arnold@univ-fcomte.fr}
\address{Laboratoire de Math\'ematiques de Besan\c con, UMR 6623, 
CNRS, Universit\'e Bourgogne Franche-Comt\'e,
25030 Besan\c{c}on Cedex, FRANCE}

\author[C. Le Merdy]{Christian Le Merdy}
\email{clemerdy@univ-fcomte.fr}
\address{Laboratoire de Math\'ematiques de Besan\c con, UMR 6623, 
CNRS, Universit\'e Bourgogne Franche-Comt\'e,
25030 Besan\c{c}on Cedex, FRANCE}

\date{\today}

\maketitle

\begin{abstract}
Let $\D$ be a Schauder decomposition on some 
Banach space $X$. We prove that if $\D$
is not $R$-Schauder, then there exists
a Ritt operator $T\in B(X)$ which is a multiplier with respect to $\D$, such that 
the set $\{T^n\, :\, n\geq 0\}$ is not $R$-bounded. Likewise we prove that there
exists a bounded sectorial operator $A$ of type $0$ on $X$ 
which is a multiplier with respect to $\D$, such that 
the set $\{e^{-tA}\, : \, t\geq 0\}$ 
is not $R$-bounded.
\end{abstract}

 
\vskip 1cm 
$R$-boundedness plays a prominent role in the study of sectorial operators and Ritt operators.
Namely the notions of $R$-sectorial operators and $R$-Ritt operators have been instrumental 
in the development of $H^\infty$-functional calculus, square function estimates and applications
to maximal regularity and to many other aspects of the harmonic analysis of semigroups
(in either the continuous or the  discrete case). 

The existence of sectorial operators which are not $R$-sectorial was discovered 
by Kalton and Lancien in their paper solving the $L^p$-maximal regularity problem \cite{KL}. 
The existence of Ritt operators which are not $R$-Ritt was established a bit later by Portal \cite{P}.
More recently, Fackler \cite{F} extended the work of Kalton-Lancien in various directions.
In contrast with \cite{KL}, which focused on existence results, \cite{F} supplied explicit constructions of 
sectorial operators which are not $R$-sectorial. Further it is easy to derive from the latter paper
explicit constructions of Ritt operators which are not $R$-Ritt. In \cite{F,KL,P}, 
sectorial operators which are not $R$-sectorial (resp. 
Ritt operators which are not $R$-Ritt) are defined as multipliers with respect to
Schauder decompositions having various ``bad" properties. In particular,
these Schauder decompositions cannot be $R$-Schauder (see Lemma \ref{R-Venni}).

The aim of this note is two-fold. First we show that given {\it any} Schauder decomposition $\D$
which is not $R$-Schauder, one can define a sectorial operator $A$ which is a multiplier with
respect to $\D$ and which is not $R$-sectorial (resp. a Ritt operator $T$ which is a multiplier with
respect to $\D$ and which is not $R$-Ritt). Second we strengthen these negative results in both
cases by showing that $A$ can be chosen bounded and such that $\{e^{-tA}\, : \, t\geq 0\}$ 
is not $R$-bounded, whereas $T$ is taken such that $\{T^n : \, n\geq 0\}$ 
is not $R$-bounded. (See Remark \ref{Rk} for more comments.)

In addition to the above mentioned papers, 
we refer the reader to \cite{B,KW,LM1,W} for relevant information on $R$-sectorial and $R$-Ritt operators.
We also mention \cite{LLM} which contains examples of Ritt operators which are not $R$-Ritt.
They are of a different nature to those in \cite{P}.

We now introduce the relevant definitions and constructions
to be used in this paper. 
Throughout we let $X$ be a complex Banach space and 
we let $B(X)$ denote the Banach algebra of all bounded operators on $X$.
We let $I_X$ denote the identity operator on $X$.

Let $(\varepsilon_j)_{j\geq 1}$ be an independent sequence of Rademacher variables on
some probability space $(\Omega, d\Pdb)$. Given any $x_1,\ldots, x_k$ in $X$, we set 
$$
\Bignorm{\sum_{j=1}^k\varepsilon_j\otimes x_j}_{R,X} = \int_\Omega
\Bignorm{\sum_{j=1}^k\varepsilon_j(u)x_j}_X\, d\Pdb(u).
$$
Then we say that a subset $F\subset B(X)$ is $R$-bounded provided that there exists a
constant $K\geq 0$ such that for any $k\geq 1$, for any 
$T_1,\ldots, T_k$ in $F$ and for any $x_1,\ldots, x_k$ in $X$,
$$
\Bignorm{\sum_{j=1}^k\varepsilon_j\otimes T_j(x_j)}_{R,X} \,\leq\, K\,
\Bignorm{\sum_{j=1}^k\varepsilon_j\otimes x_j}_{R,X}.
$$
We refer the reader to e.g. \cite[Chap. 8]{H} for basic information on $R$-boundedness.

For any $\omega\in(0,\pi)$, we let $\Sigma_{\omega}=\{\lambda\in\Cdb^*\, :\,
\vert {\rm Arg}(\lambda)\vert<\omega\}$. Let $A$ be a densely defined closed 
operator $A\colon D(A)\to X$, with domain $D(A)\subset X$. Let $\sigma(A)$ 
denote the spectrum of $A$ and let $R(\lambda,A)=(\lambda I_X -A)^{-1}$
denote the resolvent operator for $\lambda\notin\sigma(A)$. We say that $A$ 
is sectorial of type $\omega\in(0,\pi)$ if 
$\sigma(A)\subset\overline{\Sigma_\omega}$ and for any $\theta\in(\omega,\pi)$, the set
\begin{equation}\label{Sectorial}
\bigl\{ \lambda R(\lambda,A)\, :\, \lambda\in\Cdb\setminus\overline{\Sigma_\theta}\,\bigr\}
\end{equation}
is bounded. We further say that $A$ is sectorial of type $0$ if it is sectorial 
of type $\omega$ for any $\omega\in(0,\pi)$.

Note that if $A$ is sectorial of type
$\omega$ and $A$ is invertible, then $A^{-1}\in B(X)$ is sectorial of type 
$\omega$ as well. This readily follows from the fact that for any 
$\lambda\notin\overline{\Sigma_\omega}$,
we have $\lambda^{-1}\notin\overline{\Sigma_\omega}$ and
\begin{equation}\label{Inverse}
\lambda R(\lambda,A^{-1}) \,=\, I_X - \lambda^{-1} R(\lambda^{-1},A).
\end{equation}

We recall that $A$ is sectorial of type
$<\frac{\pi}{2}$ if and only if $-A$ generates a bounded analytic semigroup.
In this case, the latter is denoted by $(e^{-tA})_{t\geq 0}$.

Next we say that $A$ is $R$-sectorial of $R$-type $\omega\in(0,\pi)$ if $A$
is sectorial of type $\omega$ and for any $\theta\in(\omega,\pi)$, the set 
(\ref{Sectorial}) is $R$-bounded.

The following lemma is a straightforward consequence of (\ref{Inverse}).

\begin{lemma}\label{R-Inv}
Let $A$ be $R$-sectorial of $R$-type $\omega$ and assume that $A$ is invertible.
Then $A^{-1}$ is also $R$-sectorial of $R$-type $\omega$.
\end{lemma}
 
Let $A$ be a sectorial operator of type $<\frac{\pi}{2}$.
We recall that by \cite{W},  
$A$ is $R$-sectorial of $R$-type $<\frac{\pi}{2}$ if and only if the two sets
\begin{equation}\label{Test1}
\bigl\{e^{-tA}\, :\, t\geq 0\bigr\}
\qquad\hbox{and}\qquad
\bigl\{tA e^{-tA}\, :\, t\geq 0\bigr\}
\end{equation}
are $R$-bounded.

Let $T\in B(X)$. We say that $T$ is a Ritt operator if the two sets
\begin{equation}\label{Test2}
\bigl\{T^n\, :\, n\geq 0\bigr\}
\qquad\hbox{and}\qquad
\bigl\{nT^n(I_X-T)\, :\, n\geq 1\bigr\}
\end{equation}
are bounded.
We further say that $T$ is $R$-Ritt if these two sets are 
$R$-bounded.

These notions are closely related to sectoriality.
Indeed let $\Ddb=\{\lambda\in\Cdb\, :\, \vert \lambda\vert <1\}$ be the open unit disc.
Then
$T$ is a Ritt operator if and only if $\sigma(T)
\subset \Ddb\cup\{1\}$ and $I_X-T$ is sectorial
of type $<\frac{\pi}{2}$. Further in this case, $T$ is $R$-Ritt
if and only if $I_X-T$ is $R$-sectorial
of $R$-type $<\frac{\pi}{2}$. We refer the reader to 
\cite{B,LM1} and the references therein
for these results and various informations 
on Ritt operators and their applications.

We recall from \cite[Section 1.g]{LT} that a Schauder decomposition on $X$ is a 
sequence $\D=\{X_n\, :\, n\geq 1\}$ of closed subspaces 
of $X$ such that for any $x\in X$, there exists a unique sequence $(x_n)_{n\geq 1}$ of $X$ such that 
$x_n\in X_n$ for any $n\geq 1$ and 
$x=\sum_{n=1}^{\infty} x_n.$
For any $n\geq 1$, we let $p_n\in B(X)$ be the projection defined for $x$ as above by
$p_n(x)=x_n$. For any integer $N\geq 1$, consider their sum $P_N = 
\sum_{n=1}^N p_n\,$. This is a projection and the set
\begin{equation}\label{Schauder}
\{P_N\, :\, N\geq 1\}
\end{equation}
is bounded.

We say that $\D$ is an
$R$-Schauder decomposition if this set is actually $R$-bounded. 
Then a Schauder basis is called $R$-Schauder 
if its associated Schauder decomposition is $R$-Schauder.

Let $c=(c_n)_{n\geq 1}$ be a sequence of complex numbers. Assume that the sum
$\sum_{n=1}^{\infty}\vert c_{n} -c_{n+1}\vert\,$ is finite (in which case we  
say that the sequence has a bounded variation). Then $c$ has a limit. 
Let $\ell_c$ denote this limit and set
$$
{\rm var}(c)\,=\, \vert \ell_c\vert \, +\, \sum_{n=1}^{\infty}\vert c_{n} -c_{n+1}\vert\,.
$$
For any $x\in X$, the series $\sum_n c_n p_n(x)\,$ converges. This follows from an 
Abel transformation argument, using the boundedness of $\{P_N\, :\, N\geq 1\}$. Let $M_c\colon X\to X$
be defined by $M_c(x) =\sum_{n=1}^{\infty} c_n p_n(x)$, then we actually have
\begin{equation}\label{Mc}
M_c(x)=\,\ell_c x\,+\, \sum_{N=1}^{\infty}(c_{N} - c_{N+1})P_N(x)\,,\qquad
x\in X.
\end{equation}
This implies that
\begin{equation}\label{NormMc}
\norm{M_c}\leq {\rm var}(c)\,\sup_{N\geq 1}\norm{P_N}.
\end{equation}

Let $(a_n)_{n\geq 1}$ be
a nondecreasing sequence of $(0,\infty)$. Then we may define an operator
$A\colon D(A)\to X$ as follows. We let $D(A)$ be the space of 
all $x\in X$ such that the series $\sum_n a_n p_n(x)\,$ converges and for any
$x\in D(A)$, we set
\begin{equation}\label{A}
A(x)\,=\,\sum_{n=1}^\infty a_n p_n(x).
\end{equation}
Such operators 
were first introduced in \cite{BC, V}.
It is well-known that 
\begin{equation}\label{Spectrum}
\sigma(A)= \overline{\bigl\{a_n\,:\, n\geq 1\bigr\}}
\end{equation}
and that $A$ is a sectorial operator of type $0$ (see \cite{KL,L}). 
More precisely, for any $\lambda\in\Cdb\setminus\Rdb_+$, 
$R(\lambda,A)$ is the operator $M_{c(\lambda)}$ associated with the sequence 
$c(\lambda)=\bigl((\lambda -a_n)^{-1}\bigr)_{n\geq 1}$ and for any $\theta\in (0,\pi)$, we have
\begin{equation}\label{K-theta}
K_\theta = \sup\bigl\{\vert\lambda\vert {\rm var}(c(\lambda))\, : \, \lambda
\in\Cdb\setminus\overline{\Sigma_{\theta}}\bigr\}\, <\infty,
\end{equation}
see e.g. \cite[Section 2]{L}. 
This estimate and (\ref{NormMc}) show that $A$ is sectorial of type $0$.

We note for further use that by (\ref{Spectrum}),
the above operator $A$ is invertible.

In the sequel, any sectorial operator $A$ of this form 
will be called a $\D$-multiplier.

Likewise let $c=(c_n)_{n\geq 1}$
be a nondecreasing sequence of $(0,1)$. Then $c$ has a bounded variation,
which allows the definition of $T=M_c\in B(X)$ given by 
\begin{equation}\label{T}
T(x)\,=\,\sum_{n=1}^\infty c_n p_n(x),\qquad x\in X.
\end{equation}
It turns out that $T$ is a Ritt operator on $X$. Indeed, let $A$
be the sectorial operator (\ref{A}) associated with the sequence $(a_n)_{n\geq 1}$
defined by $a_n=(1-c_n)^{-1}$. 
Then $I_X-T=A^{-1}$ is sectorial of type $0$ and 
$\sigma(T)\subset [0,1]$, which ensures that $T$ is a Ritt operator.

In the sequel, any Ritt operator $T$ of this form 
will be called a $\D$-multiplier.

The following is well-known to specialists. 

\begin{lemma}\label{R-Venni}
Let $\D=\{X_n\, :\, n\geq 1\}$ be an $R$-Schauder decomposition on $X$.
\begin{itemize}
\item [(a)]
Any sectorial operator $A$ on $X$ 
which is a $\D$-multiplier is $R$-sectorial or $R$-type $0$.
\item [(b)] Any Ritt operator $T\in B(X)$ which is a $\D$-multiplier is $R$-Ritt. 
\end{itemize}
\end{lemma}

\begin{proof}
Let $F=\{ P_N\, :\, N\geq 1\}$. Let $A$ be given by (\ref{A}) and let $\theta\in(0,\pi)$. 
We may assume that $\lim_n a_n =\infty$.  
It follows from the above discussion
that for any $\lambda\in\Cdb\setminus\Rdb_+$,
$$
\bigr[\lambda R(\lambda,A)\bigr] (x)\,=\, \sum_{N=1}^{\infty}
\,\lambda\bigl(c(\lambda)_{N} -c(\lambda)_{N+1}\bigr)\, P_N(x),\qquad x\in X,
$$
with $c(\lambda)=\bigl((\lambda -a_n)^{-1}\bigr)_{n\geq 1}$.
This implies that
$$
\bigl\{\lambda R(\lambda,A)\, :\, \lambda\in\Cdb\setminus\overline{\Sigma_\theta}\bigr\}
\,\subset K_\theta \,\cdotp\overline{{\rm aco}}^{so}(F),
$$
where $K_\theta$ is given by (\ref{K-theta})
and $\overline{{\rm aco}}^{so}(F)$ stands for the the closure of the absolute convex 
hull of $F$ in the strong operator topology of $B(X)$. Since $F$
is $R$-bounded, $\overline{{\rm aco}}^{so}(F)$ is $R$-bounded as well, see e.g. \cite[Subsection 8.1.e]{H}. Then the set
(\ref{Sectorial}) is $R$-bounded, which shows (a).

Let $T$ be given by (\ref{T}). It follows from the above discussion that 
$T=I_X - A^{-1}$ for some  sectorial operator $A$ on $X$ 
which is a $\D$-multiplier. By (a) and Lemma \ref{R-Inv}, $A^{-1}$ is
$R$-sectorial of type $<\frac{\pi}{2}$. This entails that $T$ 
is $R$-Ritt.
\end{proof}

Our main result is the following.

\begin{theorem}\label{Construction}
Let $\D$ be a Schauder decomposition on $X$ and assume that $\D$ is not $R$-Schauder. 
\begin{itemize}
\item [(a)] There exists a sectorial operator $A$ on $X$ which is a  $\D$-multiplier, such that
the set
$$
\bigl\{e^{-tA^{-1}}\, :\, t\geq 0\bigr\}
$$
is not $R$-bounded. 
\item [(b)] There exists a Ritt operator $T\in B(X)$ which is a  $\D$-multiplier, such that the set 
$$
\bigl\{T^n\, :\, n\geq 0\bigr\}
$$
is not $R$-bounded.
\end{itemize}
\end{theorem}

\begin{proof}
We introduce $Q_N=I_X-P_N$ for any $N\geq 1$. The idea of the proof is to construct $A$ (resp. $T$) such that 
each $Q_N$ is close to $e^{-tA^{-1}}$ for some $t\geq 0$ (resp. to 
$T^n$ for some $n\geq 0$). 

Let $c=(c_n)_{n\geq 1}$ be a complex sequence with a bounded variation and let $N\geq 1$ be a fixed 
integer. For any $x\in X$, $Q_N(x)= \sum_{n=N+1}^{\infty} p_n(x)$ hence
$$
\bigl(M_c-Q_N\bigr)(x)\,=\, \sum_{n=1}^N c_n p_n(x)\, +\, \sum_{n=N+1}^{\infty} (c_n-1)p_n(x).
$$
On the one hand, we have
$$
\sum_{n=1}^N c_n p_n(x)\, =\, c_NP_N(x) + \,\sum_{n=1}^{N-1} (c_n -c_{n+1})P_n(x).
$$
On the other hand, 
\begin{align*}
\sum_{n=N+1}^{\infty} (c_n-1)p_n(x)\,
&=\, \sum_{n=N+1}^{\infty} (c_n-1)\bigl(Q_{n-1}(x) -Q_n(x)\bigr)\\
&=\, (c_{N+1} -1)Q_N(x) +\, \sum_{n=N+1}^{\infty} (c_{n+1}-c_n)Q_n(x).
\end{align*}
Let $K=\sup_{N\geq 1}\norm{P_N}$. If follows from these identities that
$$
\bignorm{M_c-Q_N}\,\leq (1+K)\biggl(\sum_{n=1}^{N-1}\vert c_{n+1} - c_n\vert\,
+\vert c_N\vert +\vert 1- c_{N+1}\vert 
+\, \sum_{n=N+1}^{\infty} \vert c_{n+1}-c_n\vert\biggr).
$$

Let $(a_n)_{n\geq 1}$ be a nondecreasing sequence of $(0,\infty)$, 
with $\lim_n a_n=\infty$, and let 
$A$ be the associated sectorial operator defined by (\ref{A}). Let $t>0$ 
and apply the above with 
$$
c_n = e^{-ta_n^{-1}},\qquad 
n\geq 1. 
$$
Then
$c_n\in(0,1)$ for any $n\geq 1$, the sequence
$(c_n)_{n\geq 1}$ is nondecreasing (hence has a bounded variation)
and $M_c=e^{-tA^{-1}}$. Further $\lim_n c_n = 1$. 
Consequently we have 
$$
\vert c_N\vert  = c_N = e^{-ta_N^{-1}},\qquad
\vert 1- c_{N+1}\vert = 1-c_{N+1} = 1 -e^{-ta_{N+1}^{-1}},
$$
$$
\sum_{n=1}^{N-1}\vert c_{n+1} - c_n\vert \,= c_N -c_1 \leq c_N = e^{-ta_N^{-1}},
$$
$$
\sum_{n=N+1}^{\infty} \vert c_{n+1}-c_n\vert \, = 1 -c_{N+1} = 1 -e^{-ta_{N+1}^{-1}},
$$
and hence
\begin{equation}\label{Norm-Diff}
\bignorm{e^{-t A^{-1}} - Q_N}\leq 2(1+K)\bigl(e^{-ta_N^{-1}} + (1-e^{-ta_{N+1}^{-1}})\bigr).
\end{equation}

We apply the above with 
$$
a_n= (n!)^3,\qquad n\geq 1.
$$ 
Next we consider the sequence $(t_N)_{N\geq 1}$ of positive integers given by we set 
$$
t_N= N(N!)^3,\qquad N\geq 1.
$$
Then by (\ref{Norm-Diff}), we have
$$
\bignorm{e^{-t_N A^{-1}} - Q_N}\leq 2(1+K)\bigl(e^{-N} + (1-e^{-\frac{N}{(N+1)^3}})\bigr)
\leq 2(1+K)\bigl(e^{-N} +N^{-2}\bigr)
$$
for any $N\geq 1$.
This estimate implies that
$$
\sum_{N=1}^{\infty} \bignorm{e^{-t_N A^{-1}} - Q_N}\,<\infty.
$$

Let $C$ be the above sum. Then for any $x_1,\ldots,x_k$ in $X$, we have
\begin{align*}
\Bignorm{\sum_{N=1}^k\varepsilon_N\otimes \bigl(e^{-t_N A^{-1}} - Q_N\bigr)(x_N)}_{R,X}\,
&\leq \,\sum_{N=1}^k \bignorm{e^{-t_N A^{-1}} - Q_N}\norm{x_N}\\ 
&\leq\, C\,\sup\bigl\{\norm{x_N}\, :\, 1\leq N\leq k\bigr\}\\
&\leq \, C\,\Bignorm{\sum_{N=1}^k\varepsilon_N\otimes x_N}_{R,X}.
\end{align*}
Hence 
$$
\Bignorm{\sum_{N=1}^k\varepsilon_N\otimes Q_N(x_N)}_{R,X}\,\leq\,
\Bignorm{\sum_{N=1}^k\varepsilon_N\otimes e^{-t_N A^{-1}}(x_N)}_{R,X}\,
+\, C\,\Bignorm{\sum_{N=1}^k\varepsilon_N\otimes x_N}_{R,X}.
$$
By assumption, the set $\{P_N\, :\, N\geq 1\}$ is not $R$-bounded, hence 
$\{Q_N\, :\, N\geq 1\}$ is nor $R$-bounded. The above estimate therefore
shows that the set $\bigl\{e^{-tA^{-1}}\, :\, t\geq 0\bigr\}$ cannot be $R$-bounded.
This proves (a).

To prove (b), we consider $T=e^{-A^{-1}}$. Then $T$ is the Ritt operator 
defined by $(\ref{T})$ for the sequence $c_n=e^{-a_n^{-1}}$. For any
$N\geq 1$, $t_N$ is an integer and $T^{t_N} =e^{-t_N A^{-1}}$. Hence 
the above argument shows that $\{T^n\, :\, n\geq 1\}$
is not $R$-bounded, which proves (b).
\end{proof}

Theorem \ref{Construction} provides 
a converse to Lemma \ref{R-Venni}, as follows.

\begin{corollary}
Let $\D$ be a Schauder decomposition on $X$. Then $\D$ is $R$-Schauder if and only if 
any sectorial operator on $X$ which is a $\D$-multiplier is $R$-sectorial, 
if and only if any Ritt operator on $X$ which is a $\D$-multiplier is $R$-Ritt.
\end{corollary}

\begin{proof}
Let $\D$ be a Schauder decomposition on $X$ which is not $R$-Schauder.
Let $A$ be verifying (a) in Theorem \ref{Construction}, and let $B=A^2$.
Then $B$ is a sectorial operator 
on $X$ which is a $\D$-multiplier. Assume that $B$ is $R$-sectorial, with some 
$R$-type $\omega\in (0,\pi)$. By 
Lemma \ref{R-Inv}, its inverse $B^{-1}$ is $R$-sectorial $R$-type $\omega$
as well. Hence by
\cite[Proposition 3.4]{KKW}, $A^{-1} =(B^{-1})^{\frac12}$ is $R$-sectorial of $R$-type $\frac{\omega}{2}
<\frac{\pi}{2}$.
This implies (see (\ref{Test1})) that the set $\{e^{-tA^{-1}}\, :\, t\geq 0\}$ is $R$-bounded,
a contradiction. Hence $B$ is not $R$-sectorial. 

Combining the above fact 
with Theorem \ref{Construction} (b) and Lemma \ref{R-Venni}, 
we deduce both `if and only if' results.
\end{proof}

It follows from \cite{F,KL} that if $X$ has an unconditional
basis and $X$ is not isomorphic to a Hilbert space, then $X$ has a Schauder
basis which is not $R$-Schauder. The above theorem therefore applies to all these spaces.

Further the arguments in \cite[Theorem 3.7 $\&$ Corollary 3.8]{KL} 
show that we actually have the following.

\begin{corollary}\label{Cor} Let $X$ be isomorphic to a separable Banach lattice and 
assume that $X$ is not isomorphic to a Hilbert space.
\begin{itemize}
\item [(a)] There exists a bounded sectorial operator $A$ of type $0$ on $X$ such that
$\{e^{-tA}\, :\, t\geq 0\}$ is not $R$-bounded.
\item [(b)] There exists a Ritt operator $T\in B(X)$ such that the set 
$\{T^n\, :\, n\geq 0\}$
is not $R$-bounded.
\end{itemize}
\end{corollary}

\begin{remark}\label{Rk} This final remark compares the above corollary with 
existing results. Let $X$ be isomorphic to a separable Banach lattice 
without being isomorphic to a Hilbert space.

\smallskip 
{\bf (1)} It follows from \cite{P} that there exists a Ritt operator 
$T\in B(X)$ such that $T$ is not $R$-Ritt. 
Recall that by definition, $T$ is not $R$-Ritt
if and only if one of the two sets in (\ref{Test2}) is not $R$-bounded.
Part (b) of Corollary \ref{Cor} strengthens \cite{P} by providing a Ritt operator $T$ on $X$
for which we know that the first of the two sets in (\ref{Test2}) is 
not $R$-bounded. This is an important step in the understanding of
the class of power bounded operators $T$ such that
$\{T^n\, :\, n\geq 0\}$ is
$R$-bounded. This class will be investigated in a future paper
(in preparation).
We refer to \cite{LM0} for the study of invertible 
operators $T\in B(X)$ such that $\{T^n\, :\, n\in\Zdb\}$ is
$R$-bounded.

\smallskip 
{\bf (2)} The existence of a sectorial operators $A$ of type $0$ on $X$ such that
$\{e^{-tA}\, :\, t\geq 0\}$ is not $R$-bounded follows from \cite{P}. 
Part (a) of Corollary \ref{Cor} shows that this can be achieved with a bounded $A$. 
We refer to \cite{A} for various results on bounded $C_0$-semigroups
$(T_t)_{t\geq 0}$ on Banach space such thatthe set  $\{T_t\, :\, t\geq 0\}$ is/is not
$R$-bounded.
\end{remark}

\bigskip
\noindent
{\bf Acknowledgements.} 
The authors were supported by the French 
``Investissements d'Avenir" program, 
project ISITE-BFC (contract ANR-15-IDEX-03).

\vskip 1cm

\ 
\end{document}